\documentclass[11pt,final]{article}
\usepackage{amsmath}
\usepackage{amssymb}
\usepackage{amscd,todonotes}
\usepackage{latexsym,microtype}
\usepackage{theorem}
\usepackage{epsfig,euscript}
\usepackage{amsfonts,nicefrac}
\usepackage{xypic,xspace,tikz}
\usepackage{bbm,ifpdf}
\usepackage[notref,notcite]{showkeys}
\usepackage{setspace,mathrsfs}
\usepackage [margin=3cm]{geometry}
\ifpdf
  \usepackage{pdfsync,hyperref}
\fi

\newtheorem{theorem}{Theorem}[section]
\newtheorem{lemma}[theorem]{Lemma}
\newtheorem{proposition}[theorem]{Proposition}
\newtheorem{corollary}[theorem]{Corollary}

\newtheorem{conjecture}[theorem]{Conjecture}

{
\theorembodyfont{\rmfamily}
\theoremstyle{plain}

\newtheorem{remark}[theorem]{Remark}
}

\newcommand{\qed}{\hfill \mbox{$\Box$}\medskip\newline}
\newenvironment{proof}{\noindent {\bf Proof:}}{\qed \par}

\newcommand{\Spec}{\operatorname{Spec}}
\newcommand{\Proj}{\operatorname{Proj}}

\renewcommand{\dim}{\operatorname{dim}}

\newcommand{\Sym}{\operatorname{Sym}}

\newcommand{\supp}{\operatorname{Supp}}
\newcommand{\init}{\operatorname{in}}

\newcommand{\Z}{\mathbb{Z}}

\newcommand{\N}{\mathbb{N}}

\newcommand{\C}{\mathbb{C}}
\newcommand{\cs}{\C^*}

\renewcommand{\a}{\alpha}
\renewcommand{\b}{\beta}
\newcommand{\g}{\gamma}
\renewcommand{\d}{\delta}
\newcommand{\e}{\epsilon}
\newcommand{\la}{\lambda}

\newcommand{\fM}{\mathfrak{M}}
\newcommand{\scrM}{\mathscr{M}}
\newcommand{\fN}{\mathfrak{N}}
\newcommand{\scrN}{\mathscr{N}}
\newcommand{\fg}{\mathfrak{g}}
\newcommand{\tn}{\mathfrak{t}^n}

\newcommand{\into}{\hookrightarrow}

\newcommand{\ith}{i^\text{th}}

\renewcommand{\H}{\operatorname{H}}
\newcommand{\IH}{\operatorname{IH}}
\newcommand{\HH}{\operatorname{HH}}
\newcommand{\IC}{\operatorname{IC}}
\newcommand{\hp}{\operatorname{HP_0}}
\newcommand{\gr}{\operatorname{gr}}

\newcommand{\cA}{\mathcal{A}}
\newcommand{\cS}{\mathcal{S}}
\newcommand{\cO}{\mathcal{O}}
\renewcommand{\cL}{\mathcal{L}}

\newcommand{\cM}{\mathcal{M}}

\renewcommand{\cR}{\mathcal{R}}
\renewcommand{\cH}{\mathcal{H}}
\newcommand{\cRbc}{\mathcal{R}^{\operatorname{bc}}}

\newcommand{\mmod}{/\!\!/\!}

\newcommand{\excise}[1]{}
\renewcommand{\and}{\qquad\text{and}\qquad}

\begin{document}
\spacing{1.2}

\noindent {\Large \bf 
Hypertoric Poisson homology in degree zero}\\

\noindent
{\bf Nicholas Proudfoot}\footnote{Supported by NSF grant DMS-0950383.}\\
Department of Mathematics, University of Oregon,
Eugene, OR 97403\\

{\small
\begin{quote}
\noindent {\em Abstract.}
Etingof and Schedler formulated a conjecture about the degree zero Poisson homology
of an affine cone that admits a projective symplectic resolution.  We strengthen 
this conjecture in general and prove the strengthened version for hypertoric varieties.
We also formulate an analogous conjecture for the degree zero Hochschild homology
of a quantization of such a variety.
\end{quote}
}

\section{Introduction}
Given a Poisson variety $\fN$ over $\C$, the {\bf degree zero Poisson homology group} $\hp(\fN)$
is defined to be the quotient of $\C[\fN]$ by the linear span of all brackets.  
If $\fN$ is affine and symplectic, then $\hp(\fN)$ is isomorphic to $\H^{\dim\fN}(\fN)$ via the map
that takes the class of a function to the de Rham class of that function times the appropriate power
of the symplectic form.  

The next interesting case is when $\fN$ is an affine cone that admits a projective symplectic
resolution $\fM$.  In this case, we may deform the map $\fM\to\fN$ to a map $\scrM\to\scrN$,
where $\scrM$ and $\scrN$ are varieties over the base $\H^2(\fM)$ with zero fibers $\scrM_0=\fM$ and 
$\scrN_0=\fN$ \cite{Namiaff, NamiaffII}.  
Over a generic element $\la\in\H^2(\fM)$, the map from $\scrM_\la$ to $\scrN_\la$ is an isomorphism
of affine varieties.  Then $\hp(\scrN)$ is a module over $\C[\H^2(\fM)]$ whose specialization at $\la$
is isomorphic to $\hp(\scrN_\la)$.  When $\la$ is generic, this is isomorphic to $\H^{\dim\scrN_\la}(\scrN_\la)\cong
\H^{\dim\scrM_\la}(\scrM_\la)\cong\H^{\dim\fM}(\fM)$.  Etingof and Schedler conjecture that
the dimension of the zero fiber $\hp(\fN)$ is also equal to that of $\H^{\dim\fM}(\fM)$; equivalently,
they conjecture that $\hp(\scrN)$ is free over $\C[\H^2(\fM)]$ \cite[1.3.1(a)]{ES11}.
They prove this conjecture for the Springer resolution and for Hilbert schemes of points on ALE spaces.

The goal of this paper is to both strengthen and prove this conjecture for hypertoric varieties, and to pose
an analogous strengthening for other projective symplectic resolutions of affine cones.
A hypertoric variety is an affine cone $\fN$ that admits a projective symplectic resolution
$\fM$, equivariant for an effective Hamiltonian action of a torus $T$, with $\dim T = \frac 1 2 \dim\fN$.
A hypertoric variety $\fN$ comes with a ``dual" hypertoric variety $\fN^!$, equipped with an action of its own
torus $T^!$; the relationship between these dual pairs has been studied in \cite{GDKD} and \cite{BLPWtorico}.
One of the first properties of a dual pair is that the cohomology group $\H^2(\fM)$ is canonically isomorphic
to the Lie algebra of $T^!$.
Our main result (Theorem \ref{main}) states that $\hp(\scrN)$ is isomorphic as a graded module over 
$\C[\H^2(\fM)]\cong \H^*(BT^!)$ to the equivariant intersection cohomology group $\IH^*_{T^!}(\fN^!)$,
where the grading in Poisson homology is induced by the conical action of the multiplicative group.
In particular, this implies that $\hp(\scrN)$ is a free module over $\C[\H^2(\fM)]$.

The relationship between a dual pair of hypertoric varieties is a special case of a relationship between
pairs of projective symplectic resolutions called {\bf symplectic duality}, which is being studied by
Braden, Licata, Webster, and the author (see \cite[1.5]{BLPWtorico} and \cite{BLPWgco}).
Examples of symplectic dual pairs, along with a conjectural extension of Theorem \ref{main} to this setting
(Conjecture \ref{general dual}), are given in Remark \ref{sd}.

We also note that Poisson homology of Poisson varieties 
is closely related to Hochschild homology of their quantizations.
More precisely, let $A$ be a quantization of $\C[\scrN]$, and let $\HH_0(A)$ be the quotient of $A$ by the linear
span of all commutators.  Then $\HH_0(A)$ is a filtered vector space such that $\gr\HH_0(A)$ admits a canonical
map from $\hp(\scrN)$; Etingof and Schedler conjecture that this map is an isomorphism \cite[1.3.3]{ES11}.
In Remark \ref{quantization} and Conjecture \ref{nc-dual}, we conjecture the appropriate analogue
of Theorem \ref{main} and Conjecture \ref{general dual} in the quantized setting.

The paper is organized as follows:  In Section \ref{hypertoric} we give a basic construction of a hypertoric
variety and its dual.  Section \ref{results} is devoted to the statement of the main theorem and associated conjectures.
We give a combinatorial presentation of $\hp(\scrN)$ in Section \ref{presentation}, which we use in Section
\ref{numerics} to prove that the main theorem holds on the numerical level (that is, we prove that the modules
are isomorphic without obtaining a canonical isomorphism).  Section \ref{canonical}, which draws heavily
on the machinery of \cite{TP08},
establishes the canonical isomorphism.\\

\noindent
{\em Acknowledgments:}
The authors is grateful to Travis Schedler for introducing him to Poisson homology and explaining
the various conjectures in \cite{ES11}.

\section{Hypertoric varieties}\label{hypertoric}
Fix a positive integer $n$, and let $V = \C^n$.  Let $T^n := (\cs)^n$ be the coordinate torus acting on $V$ in the standard
way, and let $X(T^n) \cong \Z^n$ be its character lattice.  
The vector space $V\oplus V^*$ carries a natural symplectic form, 
and the induced action of $T^n$ is Hamiltonian with moment map
$$\Phi:V\oplus V^*\to X(T^n)_\C\cong\C^n$$
given by the formula $\Phi(z,w) = (z_1w_1,\ldots,z_nw_n)$.  Let $\iota:G\into T^n$ be a connected algebraic subtorus,
and let $\iota^*:X(T^n)\to X(G)$ be the pullback map on characters.  
Assume additionally that 
$\iota$ is {\bf unimodular}, which means that for some (equivalently any) choice of basis for $X(G)$, 
all minors of $\iota^*$ belong to the set $\{-1, 0, 1\}$.
The action of $G$ on $V\oplus V^*$
is Hamiltonian with moment map $$\mu = \iota^*_\C\circ\Phi:V\oplus V^*\to X(G)_\C.$$
Fix once and for all a generic character $\theta\in X(G)$.
Let
$$\scrM := V\oplus V^*\mmod_\theta\, G = \Proj\C[z_1,\ldots,z_n,w_1,\ldots,w_n,t]^G$$
and
$$\scrN := V\oplus V^*\mmod_0\, G = \Spec\C[z_1,\ldots,z_n,w_1,\ldots,w_n]^G,$$
where $G$ acts on $t$ with weight $\theta$.  
The map $\mu$ descends to a pair of maps 
$$\pi:\scrM\to X(G)_\C\and \bar\pi:\scrN\to X(G)_\C,$$
and we let $$\fM := \pi^{-1}(0)\and \fN := \bar\pi^{-1}(0).$$
These two spaces are called {\bf hypertoric varieties} \cite{BD,Pr07}.
The variety $\fN$ is an affine cone, and $\fM$ is a projective symplectic resolution of $\fN$.

Let $T := T^n/G$; the action of $T^n$ on $V\oplus V^*$ descends to an effective Poisson 
action of $T$ on all of the aforementioned spaces.  There is an additional (non-Poisson) action of $S\cong \cs$
induced by the inverse scalar action on $V\oplus V^*$.
We have the following $S\times T$-equivariant commutative diagram, where $T$ fixes $X(G)_\C$
and $S$ acts on $X(G)_\C$ with weight $-2$.
\[\xymatrix{\fM \ar[d]\ar[rr] & & \scrM \ar[d]\ar[rr]^{\pi} & & X(G)_\C\ar[d]^{=} \\
\fN \ar[rr] & & \scrN \ar[rr]^{\bar\pi} & & X(G)_\C
}\]
The vector space $X(G)_\C$ surjects onto $\H^2(\fM)$
via the Kirwan map, and $\scrM$ is the pullback of the 
universal Poisson deformation of $\fM$ \cite{Namiaff, NamiaffII}.

\begin{remark}
Some justification is needed for the fact that in the introduction $\scrM$ and $\scrN$ were families
over $\H^2(\fM)$, whereas now we define them as families over $X(G)_\C$.
In most cases, the Kirwan map from $X(G)_\C$ to $\H^2(\fM)$ is an isomorphism, so the two definitions
agree; this fails if and only if the image of the Lie algebra map $\fg\to\tn\cong\C^n$
contains a coordinate line.  
(The kernel of the Kirwan map in all degrees is computed in different ways in
\cite[2.4]{Ko}, \cite[1.1]{HS}, and \cite[3.2.2]{Pr07}.)
In those cases where we do not get an isomorphism,
we really need the family over $X(G)_\C$ rather than over $\H^2(\fM)$ for our results to hold as stated.
\end{remark}

\begin{remark}\label{almost universal}
Even if the Kirwan map is an isomorphism in degree 2,
it may or may not be the case that $\scrN$ is the universal Poisson deformation of $\fN$.
In general, this universal deformation is parameterized by $\H^2(\fM)/W$, where $W$ is the Namikawa
Weyl group of $\fN$ (which may or may not be trivial).  
The deformation $\scrN$ is the pullback of the universal
deformation from $\H^2(\fM)/W$ to $X(G)_\C$.  We will refer to this as the {\bf quasi-universal} deformation of $\fN$.
\end{remark}

Let $(T^n)^*$ be the dual torus to $T^n$, and let $G^!\subset (T^n)^*$ be the inclusion of the connected subtorus
whose Lie algebra is perpendicular to that of $G\subset T^n$.  Fixing a generic character $\theta^!\in X(G^!)$,
we may construct new spaces
\[\xymatrix{\fM^! \ar[d]\ar[rr] & & \scrM^! \ar[d]\ar[rr]^{\pi^!} & & X(G^!)_\C\ar[d]^{=} \\
\fN^! \ar[rr] & & \scrN^! \ar[rr]^{\bar\pi^!} & & X(G^!)_\C
}\]
as above.  This diagram will be $S\times T^!$-equivariant, where $S$ is as before and $T^! := (T^n)^*/G^!$.
The relationship between $\fN$ and $\fN^!$ (and all of their associated geometry and representation theory)
was explored in detail in \cite{GDKD} and \cite{BLPWtorico} (see also Remark \ref{sd}).

\section{Results and conjectures}\label{results}
For any $\la\in X(G)_\C\cong\fg^*$, let $$\scrM_\la := \pi^{-1}(\la)\and \scrN_\la := \bar\pi^{-1}(\la).$$
The degree zero Poisson homology group $\hp(\scrN)$ is a module over $\C[X(G)_\C] \cong \Sym\fg$,
and for any $\la$, we have $$\hp(\scrN_\la) \cong \hp(\scrN)\otimes_{\Sym\fg}\C_{\la}.$$
The action of $S$ induces positive integer gradings of $\hp(\scrM)$ and $\Sym\fg$, with $\fg$
sitting in degree 2.  This grading descends to $\hp(\fN) = \hp(\scrN_0)$, but not to $\hp(\scrN_\la)$ for nonzero $\la$.

Observe that $T^!$ is canonically dual to $G$; in particular, 
$\H^*(BT^!)$ is canonically isomorphic as a graded ring to $\Sym\fg$.
Our main result is the following.

\begin{theorem}\label{main}
There is a canonical isomorphism of graded $\Sym\fg$-modules $\hp(\scrN) \cong \IH^*_{T^!}(\fN^!)$.
In particular, $\hp(\fN)$ is isomorphic as a graded vector space to $\IH^*(\fN^!)$.
\end{theorem}

\begin{remark}\label{freeness}
A combinatorial interpretation of
the intersection cohomology Betti numbers of $\fN^!$ was given in \cite[4.3]{PW07}.
In particular, they vanish in odd degree, thus
$\IH^*_{T^!}(\fN^!)$ is a free $\Sym \fg$-module.  
Theorem \ref{main} is therefore a strengthening in the hypertoric case a conjecture of Etingof and Schedler 
\cite[1.3.1(a)]{ES11},
which says (for an affine cone that admits a projective symplectic resolution) that the 
degree zero Poisson homology of the quasi-universal 
deformation is a free module over the coordinate ring of the base.
In fact, proving freeness (Corollary \ref{free}) is one of the steps toward proving Theorem \ref{main}.
\end{remark}

\begin{remark}\label{sd}
The relationship between $\fN$ and $\fN^!$ is a special case of a phenomenon called {\bf symplectic duality}
\cite[1.5]{BLPWtorico}, which relates pairs of affine cones that admit projective symplectic resolutions.  
Theorem \ref{main} invites the following generalization.
\begin{conjecture}\label{general dual}
Let $\fN$ and $\fN^!$ be a symplectic dual pair of cones $\fN$ and $\fN^!$ admitting projective symplectic resolutions
$\fM$ and $\fM^!$.
Let $\scrN$ be the quasi-universal deformation of $\fN$, and let $T^!$ be a maximal torus in the Hamiltonian automorphism group of $\fM^!$.

There exist canonical isomorphisms of graded vector spaces
$$\hp(\fN)\cong\IH^*(\fN^!)
\and
\hp(\scrN)\cong\IH^*_{T^!}(\fN^!),$$
where the second isomorphism is compatible with the module structure over $\C[\H^2(\fM)]\cong \H^*(BT^!)$.
\end{conjecture}
Note that
$\IH^*(\fN^!)\subset \H^*(\fM^!)$ is always concentrated in even degree \cite[2.5]{BLPWquant}, and therefore the
deformation coming from equivariant cohomology is always free.  Thus Conjecture \ref{general dual} would imply the
conjecture of Etingof and Schedler from Remark \ref{freeness} for any cone that has a symplectic dual.

While a general definition of a symplectic dual pair is still in preparation \cite{BLPWgco},
we already know a number of pairs of cones that should be examples (in addition to the hypertoric
examples in this paper).  For instance:
\begin{itemize}
\item If $G$ is a simple algebraic group and $G^L$ is its Langlands dual, then the nilpotent cone of $\fg$
should be dual to the nilpotent cone of $\fg^L$.
In this case, both Poisson homology and intersection cohomology are 1-dimensional \cite[1.6]{ES10}.
\item A normal slice inside the nilpotent cone to a subregular niloptent orbit in a simply-laced
simple Lie algebra $\fg$ should be dual
to the closure of the minimal nontrivial nilpotent orbit in $\fg^L\cong\fg$.  The 
Poisson homology Poincar\'e polynomial of the slice is computed in \cite{AL98} and the intersection
cohomology Poincar\'e polynomial of the orbit closure is computed in \cite[6.4.2]{MOV03}, and they agree.\footnote{In
the arXiv version \cite{MOV03}, it is 6.2.2.}
\item If $G$ and $G'$ are any two simply-laced simple algebraic groups and $\Gamma$ and $\Gamma'$ are the
corresponding finite subgroups of $\operatorname{SL}_2\C$, then 
the affinization of the moduli space of $G$-instantons on a
crepant resolution of $\C^2/\Gamma'$ should be dual to the affinization of the moduli space of
$G'$-instantons on a crepant resolution of $\C^2/\Gamma$.
The previous example is a special case where $G'$ and $\Gamma'$ are trivial.
\end{itemize}
\end{remark}

\begin{remark}\label{quantization}
Let $A$ be a quantization of $\scrN$; that is, a filtered algebra whose associated graded ring 
is isomorphic to $\C[\scrN]$ (with grading induced by the $S$-action), inducing the given Poisson structure.
Then $A$ has a central quotient $A_0$ which is a quantization of $\fN$.
The Hochschild homology
group $\HH_0(A) := A\big{/}[A,A]$ admits a filtration whose associated graded module
admits a surjection from $\hp(\scrN)$ as a graded module over $Z(A) \cong \Sym\fg$;
in particular, we also get a surjection from $\hp(\fN)$ to $\HH_0(A_0)$.
If the conjecture \cite[1.3.1(a)]{ES11} holds, then
these surjections are both isomorphisms \cite[\S 1.3]{ES11}.
The appropriate analogue of Conjecture \ref{general dual} is the following.

\begin{conjecture}\label{nc-dual}
In the situation of Conjecture \ref{general dual},
there exist canonical isomorphisms of filtered vector spaces
$$\HH_0(A_0)\cong\IH^*_S(\fN^!)\otimes_{\C[u]}\C_1
\and
\HH_0(A)\cong\IH^*_{S\times T^!}(\fN^!)\otimes_{\C[u]}\C_1,$$
where $\C_1$ is the one-dimensional module over
$\H^*(BS)\cong\C[u]$ annihilated by $u-1$.
The second isomorphism is compatible with the module structure over $Z(A)\cong \H^*(BT^!)$.
\end{conjecture}

In the hypertoric case, we make take $A$ to be
the ring of $G$-invariant differential operators on $V$, which is called the 
{\bf hypertoric enveloping algebra} \cite[5.2]{BLPWtorico}.  Since we know by Theorem \ref{main}
that $\hp(\scrN)$ is a free module, $\HH_0(A)$ must be a filtered free module whose associated graded module
is canonically isomorphic to $\hp(\scrN)$.
It should be possible to establish the canonical isomorphism of Conjecture \ref{nc-dual} using 
techniques similar to those that we use to prove Theorem \ref{main}; the main task would be to extend the techniques
of \cite{TP08} to a setting that takes the $S$-action into account.
\end{remark}

\begin{remark}\label{duals}
Let $\la\in X(G)_\C$ be generic.  In this case $\scrN_\la$ is smooth and affine,
thus
\begin{equation*}\label{deformation}
\hp(\scrN)\otimes_{\Sym\fg}\C_\la
\cong\hp(\scrN_\la) \cong \H^{\dim \scrN_\la}(\scrN_\la) \cong \H^{\dim \scrM_\la}(\scrM_\la) \cong \H^{\dim \fM}(\fM).
\end{equation*}
Here the second isomorphism is given by multiplication by the appropriate power of the symplectic form,
the third isomorphism comes from the fact that $\scrN_\la$ is isomorphic to $\scrM_\la$, and the fourth
isomorphism comes from the topological triviality of the family $\pi:\scrM\to X(G)_\C$.
On the other hand, \cite[7.21]{BLPWtorico} implies that 
$\IH^*_{T^!}(\fN^!)\otimes_{\Sym\fg}\C_\la$ is canonically {\em dual} to $\H^{\dim \fM}(\fM)$.
Comparing these two results means gives us a nondegenerate inner product on the vector space $\H^{\dim \fM}(\fM)$.

Somewhat surprisingly, this pairing depends nontrivially on $\la$.  Indeed, given 
a generic class $\la\in X(G)_\C$ with image $\bar\la\in\H^2(\fM)$,
it is straightforward to check that the pairing of the class $\bar\la^{\frac{1}{2}\dim\fM}\in\H^{\dim \fM}(\fM)$
with itself is equal to 1.  Moreover, for any nonzero complex number $c$, the pairing associated to $c\la$ is 
equal to $c^{-\dim\fM}$ times the pairing associated to $\la$.  It would be interesting to understand this family of 
inner products in more detail.

An analogue of \cite[7.21]{BLPWtorico} is part of the package for all symplectic dual pairs, so a similar 
phenomenon should arise for all of the examples mentioned in Remark \ref{sd}.
\end{remark}

\section{A presentation of \boldmath{$\hp(\scrN)$}}\label{presentation}
For each $\a\in X(T^n) \subset (\tn)^*$, consider the differential operator
$$\partial_\a: \Sym\tn\to\Sym\tn$$
defined by putting $\partial_\a x = \langle \a,x\rangle$ for all $x\in\tn$ and extending via the Leibniz rule.
It will be convenient for us to work in coordinates; we identify $X(T^n)$ with $\Z^n$ and $\Sym\tn$ with $\C[e_1,\ldots,e_n]$, and
we have $\partial_\a e_i = \a_i$.  We will be particularly interested in those operators $\partial_\a$
for which $\a$ is in the kernel of $\iota^*:X(T^n)\to X(G)$.

Consider the vector space
$$J:= \C\big\{\partial_\a e^\b\mid \a\in\ker\iota^*,\; \b\in\N^n,\; \supp(\a)\subset\supp(\b)\big\}\subset\C[e_1,\ldots,e_n],$$
where the {\bf support} of an element of $\Z^n$ is the set of coordinates where that element is nonzero.
This is not an ideal, but it is a module
over $$\Sym\fg\subset\Sym\tn = \C[e_1,\ldots,e_n]$$ because $\partial_\a x = 0$ for all $x\in\fg\subset\tn$.

\begin{remark}
For any $\a$ and $\b$, we have $\partial_{\a}+\partial_\b = \partial_{\a+\b}$.  For this reason,
we may restrict our attention in the definition of $J$ to those $\a$ which are primitive and have minimal support.
The unimodularity condition implies that for such an $\a$, $\a_i\in\{-1,0,1\}$ for all $i$.
Such an $\a$ is called a {\bf signed circuit}, and there are only finitely many of them.
\end{remark}

\begin{proposition}\label{explicit}
$\hp(\scrN)$ is isomorphic to $\C[e_1,\ldots,e_n]/J$ as a graded $\Sym\fg$-module.
\end{proposition}

\begin{proof}
Recall that $\hp(\scrN)$ is defined as the quotient of $\C[\scrN] = \C[z_1,\ldots,z_n,w_1,\ldots,w_n]^G$
by the linear span of all brackets, and consider the graded $\Sym\fg$-module homomorphism
$$\psi:\C[e_1,\ldots,e_n]\to\hp(\scrN)$$
taking $e_i$ to the class represented by $z_iw_i$.  We will show that $\psi$ is surjective with kernel $J$.

The invariant ring $\C[z_1,\ldots,z_n,w_1,\ldots,w_n]^G$ consists of all monomials $z^\b w^\d$ 
with $\b,\d\in\N^n$ and
$\iota^*(\b-\d) = 0$.  The Poisson bracket is given by the formula
$$\{p(z,w), q(z,w)\} = \sum_{i=1}^n r_i(z,w),$$
where $r_i(z,w)$ is the coefficient of $dz_i\wedge dw_i$ in the expansion of $dp(z,w)\wedge dq(z,w)$.
In particular, we have $$\{z_iw_i, z^{\b}w^{\d}\}= (\b_i-\d_i)z^{\b}w^{\d}.$$
This tells us that the class of $z^{\b}w^{\d}$ in $\hp(\scrN)$ is zero unless $\b=\d$, 
and therefore that $\psi$ is surjective.  


The remaining relations in $\hp(\scrN)$ come from brackets of the form
$$\left\{(zw)^\g z^\b w^\d, (zw)^\e z^\d w^\b\right\}$$
for some $\b,\d,\g,\e\in\N^n$ with $\iota^*(\b-\d) = 0$.  This bracket expands to
$$\sum_{i=1}^n (\b_i-\g_i)(\b_i+\g_i+\d_i+\e_i) (zw)^{\b+\d+g+\e}/(z_iw_i)
= \psi\left(\partial_{\b-\g}\, e^{\b+\g+\d+\e}\right).$$
Thus the kernel of $\psi$ is contained in $J$.

The fact that $J$ is contained in the kernel follows from unimodularity:
given a signed circuit $\a$, we can find a unique pair $\b,\g\in\N^n$
such that $\a=\b-\g$ and $\supp(\a) = \supp(\b)\sqcup\supp(\g)$.  Then every monomial
whose support contains that of $\a$ is a multiple of $e^{\b+\g}$.
\end{proof}

\begin{remark}
If the condition of unimodularity is dropped, then the last sentence of the proof of Proposition \ref{explicit}
will fail, and this will cause Theorem \ref{main} to fail, as well.
Geometrically, unimodularity ensures that $\fM$ is a manifold rather than an orbifold.
Thus Conjecture \ref{sd} is really about affine cones that admit symplectic resolutions;
orbifold resolutions are not good enough.
\end{remark}

\section{Numerics}\label{numerics}
In this section we will prove Theorem \ref{main} minus the word ``canonical".  That is, we will
show that $\hp(\scrN)$ is a free $\Sym\fg$-module with the same Hilbert series as $\IH^*_{T^!}(\fN^!)$.
This will be a necessary first step toward establishing the canonical isomorphism, which we will do in the next section.

\begin{lemma}\label{degeneration}
The graded $\Sym\fg$-module $\C[e_1,\ldots,e_n]/J$ degenerates flatly to 
$\C[e_1,\ldots,e_n]/J_{\Delta^{\operatorname{bc}}}$, where $J_{\Delta^{\operatorname{bc}}}$ is 
the Stanley-Reisner ideal of the broken circuit complex of the matroid associated to the inclusion $G\subset T^n$.
\end{lemma}

\begin{proof}
Consider the graded lexicographical term order on $\C[e_1,\ldots,e_n]$, which allows us to define the initial
$\Sym\fg$-module $$\init(J) := \{\init(f)\mid f\in J\}\subset\C[e_1,\ldots,e_n].$$  We want to show that 
$\init(J)=J_{\Delta^{\operatorname{bc}}}$.
For all $\a\in\ker\iota^*$ and $\b\in\N^n$ with $\supp(\a)\subset\supp(\b)$, $\init(\partial_\a e^\b) = \a_i\b_i e^\b/e_i$,
where $i$ is the maximal element of the support of $\a$.  Thus we have
$$\init(J) \supset \C\big\{\init(\partial_\a e^\b)\mid \a\in\ker\iota^*,\; \b\in\N^n,\; \supp(\a)\subset\supp(\b)\big\}
= J_{\Delta^{\operatorname{bc}}}.$$
We do not yet know whether or not this containment is an equality, because the set of initial terms of a basis
for a module need not form a basis for the initial module.  What we do know is that $\C[e_1,\ldots,e_n]/\init(J)$
is isomorphic to a quotient of $\C[e_1,\ldots,e_n]/J_{\Delta^{\operatorname{bc}}}$ by some $\Sym\fg$-submodule,
which we will call $Q$.  

By \cite[4.3]{PW07} and \cite[Proposition 1]{PS}, 
$\C[e_1,\ldots,e_n]/J_{\Delta^{\operatorname{bc}}}$ is a free $\Sym\fg$-module with the same Hilbert series as 
$\IH^*_{T^!}(\fN^!)$.
As noted in Remark \ref{duals}, for generic $\la\in X(G)_\C \cong \fg^*$, 
$\hp(\scrN)\otimes_{\Sym\fg}\C_\la$
is isomorphic to $\H^{\dim\fM}(\fM)$, which is in turn dual to $\IH^*_{T^!}(\fN^!)\otimes_{\Sym\fg}\C_\la$.
In particular $\hp(\scrN)\otimes_{\Sym\fg}\C_\la$ and $\IH^*_{T^!}(\fN^!)\otimes_{\Sym\fg}\C_\la$ have
the same vector space dimension, and therefore so do
$$\big(\C[e_1,\ldots,e_n]/\init(J)\big)\otimes_{\Sym\fg}\C_\la
\and
\big(\C[e_1,\ldots,e_n]/J_{\Delta^{\operatorname{bc}}}\big)\otimes_{\Sym\fg}\C_\la.$$
This implies that $Q\otimes_{\Sym\fg}\C_\la = 0$.  But since $Q$ is a submodule of a free module
and $\la$ is generic, this implies that $Q=0$, and we are done.
\end{proof}

\begin{remark}
In the last paragraph of the proof of Lemma \ref{degeneration}, we invoked the statement that
$\H^{\dim\fM}(\fM)$ is (naturally) dual to $\IH^*_{T^!}(\fN^!)\otimes_{\Sym\fg}\C_\la$.  This result, which is established
in \cite[7.21]{BLPWtorico}, builds on an enormous amount of background material, and might be frustrating
to a reader who does not want to take the time to learn all about hypertoric category $\cO$ and symplectic duality.
In fact, this was overkill; all we needed to know was that the dimension of $\H^{\dim\fM}(\fM)$ is
equal to that of $\IH^*_{T^!}(\fN^!)\otimes_{\Sym\fg}\C_\la$, which is the same as the total dimension of $\IH^*(\fN^!)$.
This fact follows from \cite[3.5 \& 4.3]{PW07}, along with the basic combinatorial fact
that the sum of the $h$-numbers of $\Delta^{\operatorname{bc}}$ is equal to the top $h$-number of the dual matroid.
\end{remark}

\begin{corollary}\label{free}
$\hp(\scrN)$ is a free $\Sym\fg$-module with the same Hilbert series as $\IH^*_{T^!}(\fN^!)$.
\end{corollary}

\begin{proof}
This follows from Proposition \ref{explicit} and
Lemma \ref{degeneration} along with \cite[4.3]{PW07} and \cite[Proposition 1]{PS}.
\end{proof}

\section{The canonical isomorphism}\label{canonical}
In this section we prove Theorem \ref{main}.
Consider the hyperplane arrangement $\cH = \{H_1,\ldots,H_n\}$ in the vector space $\fg\subset\tn$
where $H_i$ is the intersection of $\fg$ with the $\ith$ coordinate hyperplane of $\tn\cong\C^n$.  (This is the
hyperplane arrangement that is standardly associated with the hypertoric variety $\fN^!$, for example
in \cite{BD} or \cite{Pr07}.)  A {\bf flat} of $\cH$ is a subspace of $\fg$ obtained by intersecting
some (possibly empty) subset of the hyperplanes.  Let $L_\cH$ be the poset of flats of $\cH$, ordered by reverse
inclusion.  

The set $L_\cH$ has a topology in which $U\subset L_\cH$ is open if and only if 
whenever $F\leq F'$ and $F'\in U$, we have $F\in U$.  If $\cS$ is a sheaf on $L_\cH$ and $F$ is a flat,
let $\cS(F)$ be the stalk of $\cS$ at $F$.  For each $F$, there is a minimal open set $U_F$ containing $F$,
so $\cS(F)$ is simply equal to $\cS(U_F)$.  We have $U_F\subset U_{F'}$ if and only if $F\leq F'$,
therefore we have a collection of restriction maps $r(F,F'):\cS(F')\to\cS(F)$ for every pair of comparable flats.
By \cite[1.1]{TP08}, a sheaf is completely determined by its stalks and these restriction maps.
Let $\cA$ be the sheaf of algebras with $\cA(F) = \Sym(\fg/F)$, along with the obvious restriction maps.
This sheaf is called the {\bf structure sheaf} of $L_\cH$.

We will be interested in sheaves of graded $\cA$-modules on $L_\cH$.  A sheaf $\cL$ of graded $\cA$-modules is 
called a {\bf minimal extension sheaf} if it satisfies four properties:
\begin{itemize}
\item $\cL$ is indecomposable
\item $\cL$ is flabby
\item $\cL(F)$ is a free $\cL(F)$-module for all $F$
\item $\cL(\fg) \cong \cA(\fg) = \C$.
\end{itemize}
Such a sheaf exists by \cite[1.10]{TP08}.  For any two minimal extension
sheaves $\cL$ and $\cL'$, there exists an isomorphism of $\cA$-modules from $\cL$ to $\cL'$,
and this isomorphism is unique up to scalar multiplication \cite[2.7]{TP08}.
Furthermore, there is a particular minimal extension sheaf $\cL$ (defined by applying
a certain localization functor to the equivariant $\IC$-sheaf of $\fN^!$) with the property
that $\cL(0) = \IH^*_{T^!}(\fN^!)$ \cite[2.7]{TP08}.  

To prove Theorem \ref{main}, we will find another minimal extension sheaf $\cM$ with the property that 
$\cM(0)$ is canonically isomorphic to $\hp(\scrN)$.
This will get us an isomorphism between $\cL(0) = \IH^*_{T^!}(\fN^!)$ and
$\cM(0) \cong \hp(\scrN)$ of modules over $\cA(0) = \Sym\fg$.  
Since the isomorphism between $\cL$ and $\cM$ is unique up to scalar multiplication, the isomorphism
between stalks at 0 will be canonical up to scalar multiplication.
It can then be made completely canonical by noting that both $\IH^*_{T^!}(\fN^!)$
and $\hp(\scrN)$ are canonically isomorphic to $\C$ in degree zero.

Recall that the data with which we began in Section \ref{hypertoric} was an inclusion of tori $G\into T^n$,
or equivalently an inclusion of abelian Lie algebras $\fg\to\tn$.  For each flat $F$,
let $$\tn_F := \tn/\C\{e_i\mid F\not\subset H_i\},$$ and consider the inclusion $\fg/F\into\tn_F.$
With this starting point, we can repeat all of the constructions in this paper with $F$ in the subscript. 

Define two sheaves of graded $\cA$-modules $\cM$ and $\cRbc$, where
$$\cM(F):= \Sym\tn_F/J_F
\and
\cRbc(F) := \Sym\tn_F/J_{\Delta_F^{\operatorname{bc}}},$$
with the obvious restriction maps.  Lemma \ref{degeneration} tells us that $\cM$ admits a filtration whose
associated graded module is isomorphic to $\cRbc$.  We know that $\cRbc$ is a minimal extension sheaf
by \cite[3.9]{TP08}, thus the same is true of $\cM$.  This completes the proof of Theorem \ref{main}.

\begin{remark}
{\bf What just happened?}  We now summarize our approach to the proof of Theorem \ref{main},
in case the central idea got lost in the machinery of \cite{TP08}.  Given two free graded $\Sym\fg$-modules
$L$ and $M$ with the same Hilbert series, they are necessarily isomorphic, but not canonically so.
The problem is that graded $\Sym\fg$-modules are very floppy things--that is, they have lots of automorphisms.

On the other hand, a minimal extension sheaf on $L_\cH$ is a rigid object--that is, it has only scalar automorphisms--whose
space of global sections is a graded $\Sym\fg$-module.  Thus, if $L$ and $M$ can be promoted to minimal extension sheaves $\cL$ and $\cM$,
then $\cL$ and $\cM$ are canonically isomorphic (up to scalars), which induces a canonical isomorphism between $L$ and $M$
(up to scalars).  If $L$ and $M$ are both canonically isomorphic to $\C$ in degree zero, then we can do away with the scalar ambiguity.

This becomes particularly bizarre when $M$ admits a filtration such that $\gr M$ is canonically isomorphic to $L$,
and this lifts to a filtration on $\cM$ such that $\gr\cM$ is isomorphic to $\gr\cL$.  In this case, you don't really expect
$L$ and $M$ to be canonically isomorphic, but rigidity of minimal extension sheaves tells you that they are.

In our case we have three modules, $L = \IH^*_{T^!}(\fN^!)$, $R^{\operatorname{bc}} = \Sym\tn/J_{\Delta^{\operatorname{bc}}}$,
and $M = \hp(\scrN)$, along with a filtration on $M$ with $\gr M\cong R^{\operatorname{bc}}$ (Lemma \ref{degeneration}).  
The work of lifting $L$ and $R^{\operatorname{bc}}$ to sheaves $\cL$ and $\cRbc$ on $L_\cH$ and proving that these
sheaves are minimal extension sheaves is done in \cite{TP08}.  The work of lifting $M$ to a sheaf $\cM$ on $L_\cH$ is easy,
and we prove that it is a minimal extension sheaf by lifting the filtration on $M$ to one on $\cM$ with $\gr\cM\cong\cL$.

Note that we used exactly this approach in \cite{TP08} with a different module $R$, known as the {\bf Orlik-Terao algebra}
of $\cH$.  It had already been shown in \cite[Theorem 4]{PS} that $R$ admits a filtration with $\gr R\cong R^{\operatorname{bc}}$,
and it was easy to lift that to a filtered sheaf $\cR$ on $L_\cH$ with $\gr \cR\cong\cRbc$.  
It follows that $\cR$ is a minimal extension
sheaf \cite[3.11]{TP08}, and therefore that $R$ is canonically isomorphic to $L$ \cite[4.5]{TP08}.
The rest of \cite{TP08} is devoted to showing that the ring structure on $L = \IH^*_{T^!}(\fN^!)$ induced by this isomorphism
can be canonically lifted to define the structure of a ring object on the equivariant $\IC$-sheaf in the equivariant derived category of $\fN^!$
\cite[5.1]{TP08}.
\end{remark}

\bibliography{./symplectic}
\bibliographystyle{amsalpha}
\end{document}